\documentclass[11pt,a4paper]{amsart}
\pdfoutput=1
\usepackage[margin=2cm]{geometry}
\usepackage[utf8]{inputenc}
\usepackage[T1]{fontenc}
\usepackage{microtype}
\usepackage{mathtools}
\usepackage{amssymb}
\usepackage{stmaryrd}
\usepackage{verbatim}
\usepackage{tikz}
\usetikzlibrary{calc,positioning}
\usepackage{tikz-3dplot}
\usetikzlibrary{3d}
\usepackage{graphicx}
\usepackage{enumitem} \setlist{noitemsep}
\usepackage[labelformat=simple,labelfont=up]{subcaption}

\usepackage{xcolor}
\usepackage[cmtip,all]{xy}
\usepackage{thmtools}
\usepackage{chngcntr}
\usepackage{hyperref}
\hypersetup{
    pdftitle={Visual right-angled Artin subgroups of two-dimensional
  right-angled Coxeter groups},    
    pdfauthor={Christopher H. Cashen and Alexandra Edletzberger},     
    pdfkeywords={Right-angled Coxeter group, right-angled Artin group, RACG,
  RAAG, visual subgroup, commensurable, satellite-dismantlable}, 
    colorlinks=true,       
    linkcolor=black,          
    citecolor=black,        
    filecolor=black,      
    urlcolor=black           
}

\declaretheorem[
	name=Theorem,
	numberwithin=section
	]{thm}
\declaretheorem[
	name=Lemma,
	sibling=thm,
	]{lem}
\declaretheorem[
	name=Proposition,
	sibling=thm,
	]{prop}
\declaretheorem[
	name=Corollary,
	sibling=thm,
	]{cor}

\declaretheorem[
	name=Definition,
	style=definition,
	sibling=thm
	]{defin}

\declaretheorem[
	name=Coning Algorithm,
	style=remark,
	sibling=thm
	]{alg}
\declaretheorem[
	name=Global Search Algorithm,
	style=remark,
	sibling=thm
	]{galg}
\declaretheorem[
	name=Relative Search Algorithm,
	style=remark,
	sibling=thm
	]{ralg}
\declaretheorem[
	name=Example,
	style=remark,
	sibling=thm
	]{exam}

 \declaretheorem[
        numbered=no,
	name=Remark,
	style=remark,
	]{remark}

\counterwithin{figure}{section}

\DeclareMathOperator{\lk}{lk}
\DeclareMathOperator{\st}{st}
\DeclareMathOperator{\supp}{supp}
\DeclareMathOperator{\hull}{hull}
\newcommand{\diag}{\boxslash}
\newcommand{\join}{*}
\newcommand{\edge}{\mathbin{\tikz[baseline=-\the\dimexpr\fontdimen22\textfont2\relax ]{\filldraw (0,0) circle (1pt) (.2,0) circle (1pt);\draw (0,0)--(.2,0);}}}
\renewcommand{\setminus}{-}
\newcommand{\blue}{\texttt{b}}
\newcommand{\red}{\texttt{r}}
\newcommand{\CFS}{$\mathcal{CFS}$}
\newcommand{\cut}[2]{\binom{#1}{#2}}

\title[Visual RAAG subgroups of 2--dimensional RACGs]{Visual
  right-angled Artin subgroups of two-dimensional right-angled Coxeter groups}
\author[Cashen]{Christopher H.\ Cashen}
\address{Faculty of Mathematics\\University of
  Vienna\\Oskar-Morgenstern-Platz 1\\1090 Vienna, Austria\\
\href{https://orcid.org/0000-0002-6340-469X}{\includegraphics[scale=.75]{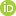}
  0000-0002-6340-469X}}
\email{christopher.cashen@univie.ac.at}
\author[Edletzberger]{Alexandra Edletzberger}
\address{Faculty of Mathematics\\University of
  Vienna\\Oskar-Morgenstern-Platz 1\\1090 Vienna, Austria\\
\href{https://orcid.org/0000-0002-6584-5149}{\includegraphics[scale=.75]{ORCID-iD_icon-16x16}
  0000-0002-6584-5149}}
\email{alexandra.edletzberger@gmail.com}

\thanks{This paper is an extension of  a section of
  the thesis of the second author \cite[Section~4.1]{thesis}, submitted in partial fulfilment of the requirements for the degree of
Doktorin der Naturwissenschaft (Dr.\ rer.\ nat.) at the University of
Vienna.
A.E.\ is grateful to Pallavi Dani for drawing her attention to the
questions addressed in this paper.
Both authors supported by the Austrian Science Fund (FWF): P34214-N} 

\begin{document}
\begin{abstract}
  There is a procedure, due to Dani and Levcovitz, for taking a finite
  simplicial graph \(\Gamma\) and a
  subgraph \(\Lambda\) of its complement, checking some
  conditions, and, if satisfied, producing a graph \(\Delta\) such that the
  right-angled Artin group with presentation graph \(\Delta\) is a
  finite index subgroup of the right-angled Coxeter group with
  presentation graph \(\Gamma\).
They do not tell us how to find \(\Lambda\),  given \(\Gamma\).

We show, in the 2--dimensional case, that the existence of such a
\(\Lambda\) is connected to the graph property of
satellite-dismantlabilty of \(\Gamma\), and we use this to give an algorithm for
producing a suitable \(\Lambda\) or deciding that one does not exist.

\end{abstract}
\vspace*{-.5cm}
\maketitle
\vspace{-.5cm}
\section{Introduction}
Every right-angled Artin group (RAAG) is a finite index subgroup of a right-angled Coxeter group (RACG), and given the presentation graph $\Delta$ of the RAAG $A_\Delta$ there is a simple graph operation that turns it into a graph $\Gamma$ that is the presentation graph of the RACG supergroup $W_\Gamma$ \cite{davisjanus}.
The converse is not true; RACGs are more varied, and there are invariants such as divergence that show that some RACGs are not even quasiisometric to a RAAG.
So, what graph conditions on $\Gamma$ imply that $W_\Gamma$ is commensurable to a RAAG?

Consider the case that $\Gamma$ is a square.
Each pair of diagonal vertices generate an infinite dihedral group, and these two dihedral groups commute.
Each of the dihedral groups has an index two, infinite cyclic subgroup, and these make an index four, $\mathbb{Z}^2$ subgroup of $W_\Gamma$.
This is the basic example of a finite index \emph{visual} RAAG subgroup; it is `visual' in the sense that we `see' the RAAG generators as pairs of non-adjacent vertices of $\Gamma$, and they commute when the vertex pairs from $\Gamma$ make a square.

This situation generalizes as follows: let $\Lambda$ be a subgraph of the complement graph $\Gamma^c$ of $\Gamma$; that is, $\Gamma^c$ has the same vertex set as $\Gamma$, and has an edge if and only if $\Gamma$ does not.
Edges of $\Lambda$ give pairs of generators of $W_\Gamma$ that generate an infinite dihedral subgroup.
Let $\Delta$ be the graph with one vertex for each edge of $\Lambda$, and an edge between two vertices if the corresponding subgroups commute, which is the case exactly when they span a square in $\Gamma$.
There is a homomorphism from $A_\Delta$ to $W_\Gamma$ given by sending a generator of $A_\Delta$ to the product in $W_\Gamma$ of the two endpoints of the corresponding edge of $\Lambda$. 
In general, however, this homomorphism is not injective, nor does it have finite index image.
Based on initial results of LaForge \cite{LaForgeVisArtin}, Dani and Levcovitz \cite{danilev24} give conditions on $\Lambda$ that determine whether the natural homomorphism is injective when $\Lambda$ has at most two connected components.
In particular, this is always sufficient \cite[Lemma~4.7]{danilev24} for the 2--dimensional case, when $\Gamma$ is triangle-free, and in this case they give necessary and sufficient conditions on $\Lambda$ for $A_\Delta$ to be a finite index subgroup of $W_\Gamma$.

We focus on the 2--dimensional case, and call a $\Lambda<\Gamma^c$ satisfying their conditions a \emph{finite index Dani-Levcovitz $\Lambda$ (FIDL--$\Lambda$)}.
Checking that a given subgraph of $\Gamma^c$ is a FIDL--$\Lambda$ is algorithmic, and 
since $\Gamma$ is finite, one can simply enumerate subgraphs of $\Gamma^c$ and check them all.
This is slow, even for rather small examples.
We are interested in a more efficient algorithm for starting from $\Gamma$ alone and either producing a FIDL--$\Lambda$ or deciding that one does not exist.

We give such an algorithm as Global Search Algorithm~\ref{SearchAlg}.
The key step, Theorem~\ref{maintheorem}, is that a FIDL--$\Lambda$ exists if and only if $\Gamma$ admits a satellite-dismantling sequence that reduces it to a square and satisfies some additional conditions that can be checked only from $\Gamma$.  
We apply this algorithm to a large number of examples via computer computations in a forthcoming work.

\section{Preliminaries} \label{SectionPreliminaries}
\subsection{Graphs}
$A\join B$ denotes the \emph{join} of $A$ and $B$; that is, the complete bipartite graph with one part the elements of $A$ and the other the elements of $B$.
A graph is \emph{complete} if for every pair of vertices there exists an edge between them. It is \emph{incomplete} if there exist two vertices that are not joined by an edge. The empty graph and a graph consisting of a single vertex are complete. 
A \emph{clique} is a complete subgraph.
A graph $\Gamma$ is \emph{separated by a clique} if there is a clique $C$ such that $\Gamma\setminus C$ has more than one connected component. A disconnected graph is separated by the empty clique. 
The \emph{link} $\lk(v)$ of a vertex $v$ in a graph is the induced subgraph on its neighboring vertices.
The \emph{star} $\st(v)$ of a vertex $v$ is $\{v\}\join\lk(v)$.
A \emph{loop} is an edge path that starts and ends at the same vertex, and a \emph{cycle} is a loop that has no repeated vertices.  
A subgraph of $\Gamma$ is \emph{induced} if it contains all of the edges between its vertices.
The \emph{induced subgraph} of a set of vertices is the induced subgraph that they span.
Given a cycle $\gamma$, an \emph{$n$--chord}, or just \emph{chord} when $n=1$, is a path of length $n$ between vertices $x$ and $y$ of $\gamma$ such that both subsegments of $\gamma$ between $x$ and $y$ have length greater than $n$. 

Distinct vertices $v$ and $w$ are \emph{twins} if $\lk(v)=\lk(w)$, and  
 $v$ is a \emph{satellite of $w$} if $\lk(v)\subset\lk(w)$.
A vertex is a \emph{satellite} if it is a satellite of some vertex. 
A graph $\Gamma$ is \emph{satellite-dismantlable to a square} if there exists a sequence $\Gamma=\Gamma_0\supset\Gamma_1\supset\dots\supset\Gamma_n$ such that $\Gamma_i\setminus\Gamma_{i+1}$ is a single satellite and $\Gamma_n$ is a square. 
This is reminiscent of  the more common graph-theoretic notion of a \emph{dismantlable graph}, in which a vertex $v$ is \emph{dominated by} $w$ if $\st(v)\subset\st(w)$, and a graph is dismantlable if it is possible to reduce it to a single vertex by removing one dominated vertex at a time.

\subsection{RACGs and RAAGs} \label{SecRACGsAndRAAGs}
See  \cite{DavisCoxeter} and  \cite[Section 2.6]{CharneyRAAGs} for background on RACGs and RAAGs.
\begin{defin}
  The \textit{right-angled Coxeter group (RACG)} $W_{\Gamma}$ defined by a finite, simplicial graph $\Gamma$ is:
 \[W_{\Gamma} = \langle s \in \Gamma  \mid s^2 = 1 \text{ for all } s \in \Gamma \, , st = ts \text{ if } (s,t) \in \mathrm{Edges}(\Gamma)\rangle\]
    The \textit{right-angled Artin group (RAAG)} $A_{\Delta}$ defined by a finite, simplicial graph $\Delta$ is:
    \[A_{\Delta} = \langle m \in \Delta \mid mn = nm \text{ if } (m,n) \in \mathrm{Edges}(\Delta)\rangle\]
The graphs are called the \emph{presentation graphs}\footnote{This is different from the conventions used to define the \emph{Coxeter graph}, which is more commonly used for not-necessarily-right-angled Coxeter groups.}.
\end{defin}

\begin{defin}
  If $W_\Gamma$ is a RACG and $\Upsilon$ is an induced subgraph of $\Gamma$, then the subgroup of $W_\Gamma$ generated by vertices of $\Upsilon$ is called a \emph{special subgroup}.
It is a RACG with presentation graph $\Upsilon$, and is denoted $W_{V(\Upsilon)}$ or $W_\Upsilon$.
The analogous statement and terminology also applies to RAAGs.
\end{defin}

We will restrict to one--ended groups.
A RACG is one--ended if and only if it is incomplete and has no separating clique \cite[Theorem 8.7.2]{DavisCoxeter}.
A RAAG is one-ended if it is connected and has at least two vertices.
To further simplify the set-up, we consider only RACGs and RAAGs whose Davis and Salvetti complexes, respectively, are two-dimensional. This is satisfied for both RACGs and RAAGs if the presentation graph is triangle-free.

There is a quasiisometry invariant known as \emph{divergence}.
In particular, if a group has polynomial divergence then the degree of the polynomial is a quasiisometry invariant.
RAAGs have at most quadratic divergence \cite[Corollary 4.8]{BehrstockCharneyDivRAAGs}, thus so does every group quasiisometric  to a RAAG.
By work of Dani and Thomas \cite[Theorem 1.1]{DaniThomasDiv} a one-ended, two-dimensional RACG has at most quadratic divergence if and only if its presentation graph has a property known as $\mathcal{CFS}$ (component of full support/constructed from squares). This was generalized to higher dimension in \cite[Definition 1.3]{BehrstockCFS}.

\begin{defin} \label{defCFS}
 The \emph{diagonal graph} $\diag(\Gamma)$ of $\Gamma$ is the graph whose vertices are diagonals of induced squares in $\Gamma$, with $\{a,b\}$ and $\{c,d\}$ connected by an edge if $\{a,b\}\join\{c,d\}$ is an induced square in $\Gamma$.

  The \emph{support} $\supp(\{a,b\})$ of a vertex $\{a,b\}$ of $\diag(\Gamma)$ is the pair of vertices $\{a,b\}$ in $\Gamma$.
  The support of a subset of $\diag(\Gamma)$ is the union of the supports of its vertices. 

  The graph $\Gamma$ is \CFS\ if $\diag(\Gamma)$ has a connected component whose support is all non-cone vertices of $\Gamma$. If $\Gamma$ is triangle-free and not a star then it has no cone vertices, and we can simply say that some component of $\diag(\Gamma)$ has full support, ie, its support is all of $\Gamma$.

  The graph $\Gamma$ is \emph{strongly \CFS} if it is \CFS\ and $\diag(\Gamma)$ is connected. 
\end{defin}

\begin{remark}
  The usual definition of \CFS\ uses a graph $\square(\Gamma)$ whose vertices are induced squares of $\Gamma$, with an edge between two vertices if they intersect in a diagonal.
  The graphs $\square(\Gamma)$ and $\diag(\Gamma)$ carry the same information, but $\diag(\Gamma)$ is topologically simpler, since many squares intersecting in a common diagonal form a star in $\diag(\Gamma)$ but a clique in $\square(\Gamma)$. The diagonal graph is also more natural for our purposes because when we have $\Gamma$ with a FIDL--$\Lambda$, then the commuting graph $\Delta$ of $\Lambda$ sits as a subgraph in $\diag(\Gamma)$. 
\end{remark}

\begin{thm}[{\cite{DaniThomasDiv}}]\label{weaklycfs}
  If $\Gamma$ is an incomplete, triangle--free graph without separating cliques such that $W_\Gamma$ is quasiisometric to a RAAG then $\Gamma$ is \CFS.  
  \end{thm}
  \begin{lem}\label{cfsimpliesoneended}
    If $\Gamma$ is incomplete, triangle-free, and \CFS\, then it has no separating clique.
  \end{lem}
  \begin{proof}
    Take $a,b\in\Gamma$.
    If $a$ and $b$ are the diagonal of some square then they are not separated by a clique.
    Otherwise, 
    there is a path $\{p_0,q_0\},\dots,\{p_n,q_n\}$ in the full support component of $\diag(\Gamma)$ with $n>0$, $p_0=a$, and $p_n=b$.
    This path corresponds to a chain of squares $\{p_i,q_i\}\join\{p_{i+1},q_{i+1}\}$ in $\Gamma$, with successive squares sharing a diagonal. The union of the squares is not separated by a clique.  
  \end{proof}

\subsection{Dani-Levcovitz conditions} \label{DLAlgorithm}
Let $\Theta = \Theta(\Gamma,\Lambda)$ be the graph with vertex set $\Gamma$, with edges from both $\Gamma$ and $\Lambda<\Gamma^c$.
The edges coming from $\Gamma$ are \emph{$\Gamma$--edges}, and the edges coming from $\Lambda$ are \emph{$\Lambda$--edges}. Similarly, a path consisting only of $\Gamma$--edges is a $\Gamma$--path, etc.
The $\Lambda$--hull, $\hull_\Lambda$, of a subset of vertices of $\Theta$ is the vertex set of their convex hull in $\Lambda$.
A set of vertices is $\Lambda$--convex if it is equal to its $\Lambda$--hull.

Dani and Levcovitz \cite{danilev24} give \emph{subgroup conditions} $\mathcal{R}_1$--$\mathcal{R}_4$ to determine that the RAAG $A_\Delta$ on the commuting graph $\Delta$ associated to $\Lambda$ is a visual RAAG subgroup of $W_\Gamma$.
They give \textit{index conditions} $\mathcal{F}_1$ and $\mathcal{F}_2$  to ensure that the visual RAAG subgroup is of finite index in $W_\Gamma$.
They show in the 2--dimensional case that it always suffices to find $\Lambda$ with two components, and for two component $\Lambda$ their conditions are necessary and sufficient.

The conditions are as follows, simplified by specializing to the case that $\Gamma$ is an incomplete, triangle-free graph without separating cliques.
Let $\Lambda_\red$ (red) and $\Lambda_\blue$ (blue) be disjoint, connected subgraphs of $\Gamma^c$, with $\Lambda=\Lambda_\red\sqcup\Lambda_\blue$.
\begin{itemize}
\item[$\mathcal{R}_1$:] $\Lambda_\red$ and $\Lambda_\blue$ are trees.
\item[$\mathcal{R}_2$:] $\Lambda_\red$ and $\Lambda_\blue$ are induced subgraphs of $\Theta$.
\item[$\mathcal{F}_1$:] $\Lambda$ spans $\Gamma$. 
\end{itemize}
These conditions are true if and only if $\Gamma$ is bipartite, with a bicoloring $\red/\blue$ (every vertex is colored either $\red$ or $\blue$, and adjacent vertices have different colors), and  $\Lambda_\red$ and $\Lambda_\blue$ are trees in $\Gamma^c$ spanning the $\red$ and $\blue$ parts, respectively.
We will not state $\mathcal{F}_2$. In our case it is always satisfied if $\mathcal{R}_2$ and $\mathcal{F}_1$ are \cite[Remark~4.3]{danilev24}.
Assuming these conditions, we can state the remaining two conditions in simplified form:
\begin{itemize}
  \item[$\mathcal{R}_3$:] If $\{a,b\}\join\{c,d\}$ is a square in $\Gamma$ then $\hull_\Lambda\{a,b\}\join\hull_\Lambda\{c,d\}\subset \Gamma$.
\item[$\mathcal{R}_4$:] If $a\edge b$ is an edge in a cycle $\gamma$ then there is a square $\{a,a'\}\join\{b,b'\}$ with $a',b'\in\hull_\Lambda(\gamma)$.
\end{itemize}

Notice that the assumption that $\Gamma$ is incomplete with no separating clique implies $W_\Gamma$ is 1--ended, so $A_\Delta$ is 1--ended, so $\Delta$ is connected and has more than one vertex. Thus, every edge of $\Lambda$ is a diagonal of a square in $\Gamma$, since otherwise it would have nothing to commute with, so would give an isolated vertex in $\Delta$. 
Thus, we may identify $\Delta$ with a subgraph of $\diag(\Gamma)$.

\bigskip

One nice application of these conditions in \cite{danilev24} is to connect them to conditions given by Nguyen and Tran \cite{NguyenTranPlanar} on deciding when a \emph{planar} graph $\Gamma$ defines a RACG that is quasiisometric to a RAAG.
The conclusion is that, for planar $\Gamma$,  $W_\Gamma$ being quasiisometic to a RAAG implies graph conditions that imply $\mathcal{R}_1-\mathcal{R}_4$ and $\mathcal{F}_1$ and $\mathcal{F}_2$, so $W_\Gamma$ actually has a finite index visual RAAG subgroup, which happens always to be defined by a tree $\Delta$. 
Dani and Levcovitz also give two families of non-planar graphs to which their conditions apply and yield $\Delta$ that are not trees.
We mention one of these families here:

\begin{exam} \label{ExampleCubeWithDiagonal}
  A \emph{bicycle wheel} is a graph consisting of adjacent vertices $x$ and $y$, the `hub', a circle of even length $2n\geq 6$ given by $c_1$, $d_1$, $c_2$,\dots, $d_n$, the `rim', and edges from each $c_i$ to $x$, and from each $d_i$ to $y$, the `spokes'.
  \begin{figure}[h]
    \centering
    \rotatebox{0}{\tdplotsetmaincoords{70}{0}
    \begin{tikzpicture}[tdplot_main_coords, scale=1]
      \pgfmathsetmacro{\r}{2}
      \pgfmathsetmacro{\h}{.5}
      \pgfmathsetmacro{\n}{11}
            \tdplotsetcoord{O}{0}{0}{0}
      \tdplotsetcoord{h0}{\h/2}{0}{0}
      \tdplotsetcoord{h1}{\h/2}{180}{0}
      \begin{scope}\foreach \i in {0,...,\n}{
          \tdplotsetcoord{a}{\r}{90}{2*\i*180/\n}{90}
          \tdplotsetcoord{b}{\r}{90}{(2*\i+1)*180/\n}
          \tdplotsetcoord{c}{\r}{90}{(2*\i+2)*180/\n}
          \filldraw (a) circle (1pt);
          \filldraw (b) circle (1pt);
          \draw (a)--(h0);
          \draw (b)--(h1);
        }\end{scope}
 \tdplotdrawarc{(O)}{\r}{0}{360}{}{}
\draw (h0)--(h1);
\filldraw (h0) circle (1pt);
\filldraw (h1) circle (1pt);
\end{tikzpicture}}
    \caption{A bicycle wheel}
    \label{fig:bicyclewheel}
  \end{figure}
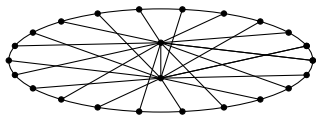

A bicycle wheel admits a 2--component FIDL--$\Lambda$ consisting of the opposite of the spokes: one star consisting of an edge from $x$ to each $d_i$ and another consisting of an edge from $y$ to each $c_i$.
The commuting graph $\Delta$ is a circle of the same length as the rim. 

In Figure~\ref{fig:CubeWithDiag} the case $n=3$ is also recognizable as the 1--skeleton of a 3--cube with one space diagonal:
    \begin{figure}[ht!]
        \centering
        \scriptsize
        \def\svgwidth{200pt}
        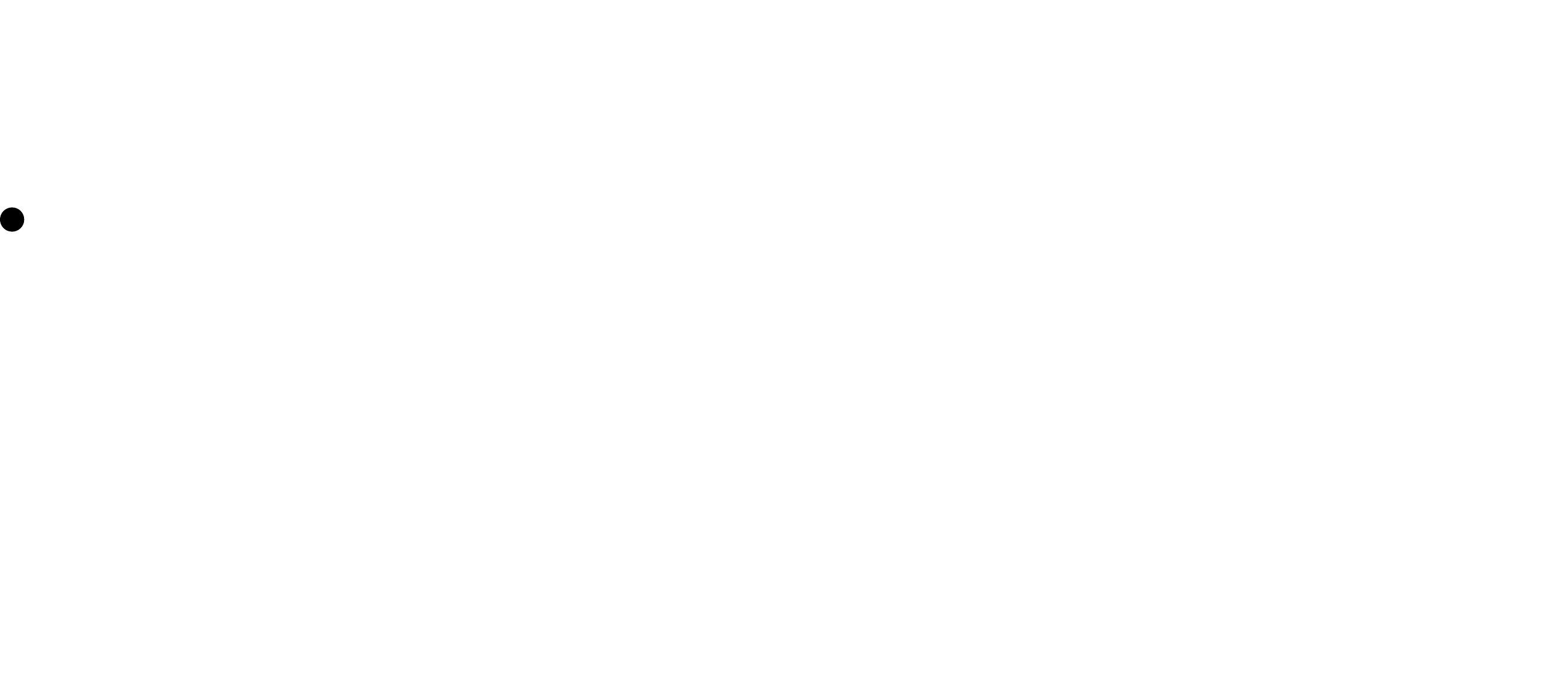
        \caption{}
        \label{fig:CubeWithDiag}
    \end{figure}
\end{exam}

\subsection{Splittings of RACGs}
We have mentioned that $W_\Gamma$ is one--ended when it is incomplete with no separating clique.
This corresponds to not having a splitting as an amalgamated product over a finite group.
The next simplest splittings are over two-ended, or virtually $\mathbb{Z}$, groups. 

A \emph{JSJ decomposition} of a finitely presented group is a certain maximal graph of groups decomposition (see \cite{JSJDecompGps} for the precise definition) with two-ended edge groups and vertex groups in three categories: two-ended, \textit{hanging} or \textit{rigid}. A hanging vertex group is essentially the fundamental group of a surface with boundary and a rigid vertex group does not split any further with respect to its incident edge groups.
JSJ decompositions are not unique, but there is a way to encode all of them simultaneously in a 
\textit{JSJ graph of cylinders}. It is the canonical representative for the deformation space of JSJ decompositions of the group and it can be used to deduce quasiisometry invariants, by \cite{CashenMartinStructureInvar}.
The idea is that some of the two--ended edge and vertex groups in a JSJ decomposition may be commensurable, and these can be grouped together to form \emph{cylinders}, and from this a new decomposition is derived.
For RACGs, all of this is visible in the presentation graph: 
Mihalik and Tschantz \cite{MihalikTschantzVisual} show a one-ended, two-dimensional RACG $W_\Gamma$ admits a splitting over a two-ended subgroup if and only if $\Gamma$ has a \emph{cut} $\cut{a}{b}$:
\begin{defin}\label{def:cut}
  If $\Gamma$ is an incomplete, triangle-free graph without separating cliques, a \emph{cut} $\cut{a}{b}$ means one of the following, both of which have the property that  the element $ab\in W_\Gamma$ generates an infinite cyclic subgroup that is finite index in $W_{\cut{a}{b}}$.
  \begin{itemize}
  \item A \emph{cut pair} $\cut{a}{b}=\{a,b\}$:  a pair of non-adjacent vertices such that $\Gamma\setminus\{a,b\}$ is not connected.
    \item A \emph{2--path cut triple} $\cut{a}{b}=\{a,b,c\}$: a triple of vertices with $c\in\lk(a)\cap\lk(b)$ such that $\Gamma\setminus\{a,b\}$ is connected but $\Gamma\setminus\{a,b,c\}$ is not. 
  \end{itemize}
\end{defin}

Definition~\ref{def:cut} implies that every component of $\Gamma\setminus\cut{a}{b}$ contains a neighbor of each vertex in $\cut{a}{b}$.

A cut pair $\{a,b\}$ is \textit{crossed} by another, disjoint, cut pair $\{c,d\}$ if $a$ and $b$ lie in different connected components of $\Gamma \setminus \{c,d\}$.
A cut $\cut{a}{b}=\{a,b,c\}$ is \textit{crossed} by a cut $\cut{d}{e}=\{d,e,f\}$ if $c$ is equal to $f$ and $a$ and $b$ lie in different connected components of $\Gamma \setminus \{d,e,c\}$.
A cut  that is not crossed by any other cut is \textit{uncrossed}.
Crossing cuts are responsible for hanging vertices in the JSJ decomposition, but RAAGs do not have these \cite{margolis2018quasi}, so they cannot appear in groups quasiisometric to RAAGs.

\begin{thm}[{\cite[Theorem~3.29]{edletzberger2021quasi}}]\label{ThmJSJGOCforRACGs}
  Let $\Gamma$ be an incomplete, triangle free graph without separating cliques.
  If the JSJ graph of cylinders of $W_\Gamma$ has no hanging vertices, it consist of:
\begin{itemize}
    \item For every pair $\{a,b\}$ such that there is an uncrossed cut $\cut{a}{b}$, there is a cylinder vertex with vertex group $W_{\{a,b\} \cup (\lk(a)\cap\lk(b))}$.
    \item For every set $B$ of essential (valence at least 3) vertices in $\Gamma$ satisfying the following conditions,  there is a rigid vertex with vertex group $W_B$:
        \begin{itemize}[leftmargin=1.5cm]
            \item[(B1)] No cut separates $B$.\label{B1n}
            \item[(B2)] The set $B$ is maximal among all sets satisfying \hyperref[B1n]{(B1)}. \label{B2n}
            \item[(B3)] $|B| \geq 4$. \label{B3n}
        \end{itemize}
\end{itemize}
Furthermore a pair of vertices is connected by an edge if and only if the pair consists of a cylinder vertex and a rigid vertex whose vertex groups intersect in a subgroup containing the two-ended cut defining the cylinder. The edge group is the intersection of its vertex groups.
\end{thm}

\section{FIDL--\texorpdfstring{$\Lambda$}{Lambda} Convexity}\label{SecConvexity}

  If $\Lambda$ is a forest and $x$ and $y$ are in the same component,
  let $[x,y]_\Lambda$  denote the unique $\Lambda$--geodesic joining them.

\begin{lem}\label{lem:componentscontaincommonneighbors}
  Let $\Gamma$ be an incomplete, triangle-free graph with no separating clique that
  admits a FIDL--$\Lambda$.
  Let $\cut{v}{v'}$ be a cut of $\Gamma$.
  Every component of $\Gamma\setminus \cut{v}{v'}$ contains a common
neighbor of $v$ and $v'$.
\end{lem}
\begin{proof}
  Pick a component of $\Gamma\setminus \cut{v}{v'}$; let $b$ be a
  vertex in that component, and let $a$ be a vertex from a different
  component.   
By Theorem~\ref{weaklycfs} we can choose a geodesic $\{p_0,q_0\},\dots,\{p_n,q_n\}$ in $\diag(\Gamma)$
such that $a\in\{p_0,q_0\}$ and $b\in\{p_n,q_n\}$.
Let $i_0$ be the least index such that $\{p_{i_0},q_{i_0}\}$ contains
a vertex, say $p_{i_0}$, in the same component of  $\Gamma\setminus \cut{v}{v'}$ as $b$.

If $i_0=0$ then $q_{i_0}=a$ and there is a square with one diagonal
containing vertices in different components of the cut.
This is only possible if the cut $\cut{v}{v'}$ is a cut pair $\{v,v'\}$
and the square is $\{a,p_0\}\join \{v,v'\}$, in which case $p_0$ is a
common neighbor of $v$ and $v'$ in the $b$--component.

If $i_0>0$ then $\{p_{i_0-1},q_{i_0-1}\}\join\{p_{i_0},q_{i_0}\}$ is a
square with $p_{i_0}$ in the $b$--component 
and with $p_{i_0-1}$ and $q_{i_0-1}$ non-adjacent vertices that are
both adjacent to $p_{i_0}$, but neither of which are in the same component
of $\Gamma\setminus \cut{v}{v'}$ as $p_{i_0}$. Then
$\{p_{i_0-1},q_{i_0-1}\}=\{v,v'\}$.
\end{proof}

\begin{lem}\label{lem:cutsconvex}
  Let $\Gamma$ be an incomplete, triangle-free graph with no separating clique that
  admits a FIDL--$\Lambda$.
  Let $\cut{v}{v'}$ be a cut of $\Gamma$.
  Then $\Lambda$ contains an edge between $v$ and $v'$.
\end{lem}
\begin{proof}
  By Lemma~\ref{lem:componentscontaincommonneighbors}, it is possible
  to choose vertices $a$ and $b$ that are common neighbors of $v$ and
  $v'$ and contained in different components of $\Gamma\setminus\cut{v}{v'}$.
  Every common neighbor of $a$ and $b$ must lie in the cut, so the
  triangle-free condition implies the only common neighbors of $a$ and
  $b$ are $v$ and $v'$. 
  
  Now, $\{a,b\}\join\{v,v'\}$ is a square, so $\mathcal{R}_3$ implies
  $\{a,b\}\join\hull_\Lambda\{v,v'\}\subset\Gamma$, so
    $\hull_\Lambda\{v,v'\}=\{v,v'\}$, so there is a $\Lambda$--edge
    between $v$ and $v'$.
  \end{proof}
  
\begin{lem}\label{lem:linksconvex}
  Let $\Gamma$ be an incomplete, triangle-free graph with no separating clique that
  admits a FIDL--$\Lambda$.
  The link of every vertex is $\Lambda$--convex.
\end{lem}
\begin{proof}
  Suppose not.
  Then there exist $a,b,v$ with $a,b\in\lk(v)$ such that
  $\lk(v)\cap \hull_\Lambda\{a,b\}=\{a,b\}$ and $a$ and $b$ are not
  adjacent in $\Lambda$. 
  Assume $v\in\Lambda_\red$ and $a,b\in\Lambda_\blue$. 
Let  $a=c_0,c_1,\dots,c_n=b$, $n\geq 2$, be the vertices of
$[a,b]_\Lambda$.
Each pair $\{c_i,c_{i+1}\}$ is the
  diagonal of some square $\{c_i,c_{i+1}\}\join\{d_i,d_i'\}$ of
  $\Gamma$.
  None of the $d_i$ and $d_i'$ equal $v$, since the $c_i$
  for $i\neq 0,n$ are not in $\lk(v)$.

We build a cycle $\gamma$ as follows: start with $b$, $v$,
$a=c_0$, and $d_0$.
Next add $c_{j_0}$ where $j_0\geq 1$ is the largest index such that $c_{j_0}$ is
adjacent to $d_0$.  Then add $d_{j_0}$.
Continue, where, having most recently added $d_i$ we next add $c_j$ such that $j$ is the maximal index with $d_i$ adjacent to $c_j$.  
If $j<n$ then add $d_j$ and repeat.
The point is that while $d_{j}$ is adjacent to $c_{j+1}$, the previous $d_i$
that occur in $\gamma$ are not adjacent to any $c_k$ for $k>j$, so
we guarantee that no $d$ is repeated in $\gamma$.

Thus, $\gamma$ is a cycle in $\Gamma$ with
$\hull_{\Lambda_\blue}(\gamma)=\hull_{\Lambda}\{a,b\}=\{c_0,\dots,c_n\}$.
Condition $\mathcal{R}_4$ implies there exists a square
$\{a,a'\}\join\{v,v'\}$ with
$a',v'\in\hull_\Lambda\gamma$.
However, $a$ and $b$ are the only vertices of
$\lk(v)\cap\hull_\Lambda\gamma$, so $a'=b$.
Condition $\mathcal{R}_3$ implies
$\{v\}\join\hull_\Lambda\{a,b\}\subset\Gamma$, contradicting $\hull_\Lambda\{a,b\}\not\subset\lk(v)$.
\end{proof}

\begin{cor}\label{lem:componentsconvex}
  Let $\cut{v}{v'}$ be a cut of $\Gamma$, and let $\Gamma'$ be a connected
  component of $\Gamma\setminus \cut{v}{v'}$.
  Let $\bar\Gamma':=\Gamma'\cup \cut{v}{v'}$.
  The intersection of each component of $\Lambda$ with $\bar\Gamma'$ is $\Lambda$--convex.
\end{cor}
\begin{proof}
For every  pair of vertices $a,b\in\bar\Gamma'$, there exists an embedded
  $\Gamma$--path $c_0=a,\dots,c_n=b$ such that $c_i\in\Gamma'$ for all
  $i\neq 0,n$.
  Suppose $a$ and $b$ are in the same component of $\Lambda$.
  Then $n$ is
  even and for $i$ odd, $\lk(c_i)\subset \bar\Gamma'$ is $\Lambda$--convex and contains
  $c_{i-1}$ and $c_{i+1}$, so
  $[c_0,c_2]_\Lambda+\cdots+[c_{n-2},c_n]_\Lambda$ is a
  $\Lambda$--path from $a$ to $b$ with vertices in
  $\bigcup_{i \text{ odd}}\lk(c_i)\subset \bar\Gamma'$.
\end{proof}

\begin{cor}\label{cor:linkintersectionconvex}
     For any two vertices $a,b\in\Gamma$, $\lk(a)\cap\lk(b)$ is $\Lambda$--convex.
 \end{cor}
 \begin{proof}
For $c,d\in\lk(a)\cap\lk(b)$, 
Lemma~\ref{lem:linksconvex} says $[c,d]_\Lambda$  is contained in both $\lk(a)$ and
$\lk(b)$.
 \end{proof}

 Here are some consequences of these convexity results:
 \begin{prop}\label{nolongcycles}
    Let $\Gamma$ be an incomplete, triangle-free graph with no separating clique that
    admits a FIDL--$\Lambda$.
    Then every cycle of $\Gamma$ has even length, every cycle of
    length greater than 6 has a 1 or 2--chord, and an induced cycle of
    length 6 occurs only as
the rim of a bicycle wheel subgraph of $\Gamma$.
  \end{prop}
  \begin{proof}
    $\Gamma$ is bipartite, since its vertices are 2--colored
    $\red/\blue$ according to which component $\Lambda_\red$ or
    $\Lambda_\blue$ of $\Lambda$ they belong. Thus, $\Gamma$ has no
    odd cycles.
    Suppose  $\gamma:=c_0,\,c_1,\,\dots,\,c_{n-1}$ is a cycle of
    length $n>4$ with no 1 or 2--chords. 
    We always take subscripts
    modulo $n$, without further comment.

Construct a $\Lambda$--loop at $c_0$ by taking $[c_0,c_2]_\Lambda+[c_2,c_4]_\Lambda+\cdots+[c_{n-2},c_n]_\Lambda$.
    This is a loop in a tree, so it is degenerate.
    In particular, each edge is crossed an even number of times.
    By Lemma~\ref{lem:linksconvex}, $[c_{2m},c_{2m+2}]_\Lambda\subset \lk(c_{2m+1})$.
    Since the cycle has no 1 or 2--chords, for odd $j>i$ we have that $\lk(c_{i})$ and
    $\lk(c_{j})$ intersect only if $j=i+2$ or $i=1$ and $j=n-1$, so
    only $[c_{2m-2},c_{2m}]_\Lambda$ and $[c_{2m+2},c_{2m+4}]_\Lambda$ potentially share edges with 
    $[c_{2m},c_{2m+2}]_\Lambda$.
    Since a geodesic uses an edge once or not at all, 
    to have all of the edges of
    $[c_{2m},c_{2m+2}]_\Lambda$ crossed evenly in total by the loop we
    need that each edge of  $[c_{2m},c_{2m+2}]_\Lambda$  is also contained in exactly one of $[c_{2m-2},c_{2m}]_\Lambda$ and $
    [c_{2m+2},c_{2m+4}]_\Lambda$.
    Thus, there is $x\in [c_{2m},c_{2m+2}]_\Lambda$ such that
    $[c_{2m},x]_\Lambda=[c_{2m-2},c_{2m}]_\Lambda\cap
    [c_{2m},c_{2m+2}]_\Lambda$ and $[x,c_{2m+2}]_\Lambda=[c_{2m},c_{2m+2}]_\Lambda \cap
    [c_{2m+2},c_{2m+4}]_\Lambda$.
    Such an $x$ is in
    $\lk(c_{2m-1})\cap\lk(c_{2m+1})\cap\lk(c_{2m+3})$.
    But then $d_\Gamma(c_{2m-1},c_{2m+3})=2$, so, since $\gamma$ has
    no 2--chords, $d_\gamma(c_{2m-1},c_{2m+3})=2$, so $n=6$. 

The same argument, reversing evens and odds, shows
there exists  $y\in\lk(c_0)\cap\lk(c_2)\cap\lk(c_4)$.    

Consider the cycle $\gamma':=y,c_0,c_1,x,c_3,c_4$.
    By condition $\mathcal{R}_4$, there is a square
    $\{c_0,v\}\join\{c_1,w\}$ with
    $v,w\in\hull_\Lambda(\gamma')=\{x,c_0,c_4\}\sqcup\{y,c_1,c_3\}$, so
    $v\in\{x,c_4\}$ and $w\in\{y,c_3\}$.
    But since $\gamma$ has no chords, $c_1$ is
    not adjacent to $c_4$ and $c_3$ is not adjacent to $c_0$, so $v=x$
    and $w=y$, implying $x$ and $y$ are adjacent. Thus, $\gamma$ is
    the rim of a bicycle wheel with hub $\{x,y\}$.
  \end{proof}

  \begin{prop}\label{prop:splitting}
    Let $\Gamma$ be an incomplete, triangle-free graph with no
    separating clique.
    Suppose $\Gamma$ has an uncrossed cut $\cut{a}{b}$.
    Let $\Gamma_i$ be the components of $\Gamma\setminus\cut{a}{b}$,
    and let $\bar\Gamma_i:=\Gamma_i\cup\cut{a}{b}$. 
    Then $\Gamma$ admits a FIDL--$\Lambda$ if and only if each
    $\bar\Gamma_i$ admits a FIDL--$\Lambda_i$ that contains an edge 
  $a\edge b$.
  \end{prop}
  \begin{proof}
    The `only if' direction follows easily from
    Lemma~\ref{lem:cutsconvex} and
    Corollary~\ref{lem:componentsconvex}. 

 For the converse, suppose $\Lambda_{\Gamma_i}:=\Lambda_{\Gamma_i,\red}\cup
  \Lambda_{\Gamma_i,\blue}$ contains an edge $a\edge b$ in
  $\Lambda_{\Gamma_i,\red}$ for all $i$.
  This is the only intersection of any of the $\red$ trees, since
  $\cut{a}{b}=\cap_i\bar\Gamma_i$, so
  $\Lambda_\red:=\cup_i\Lambda_{\Gamma_i,\red}$ is
  a tree.

  For $\Lambda_\blue$ there are two cases.
  If $\cut{a}{b}=\{a,b\}$ is a cut pair then the
  $\Lambda_{\Gamma_i,\blue}$ are disjoint.
  For each $i$, choose  $c_i\in\Gamma_i\cap\lk(a)\cap\lk(b)$, which
  exists by Lemma~\ref{lem:componentscontaincommonneighbors}.
Define 
  $\Lambda_\blue$ as the union of the $\Lambda_{\Gamma_i,\blue}$,
  together with a tree spanning the $c_i$.
   If $\cut{a}{b}=\{a,b,c\}$ is a
  2--path cut triple then $c$ is a vertex in all 
  $\Lambda_{\Gamma_i,\blue}$, so
  $\Lambda_\blue:=\cup_i\Lambda_{\Gamma_i,\blue}$
  is a tree. In this case, set each $c_i:=c$, so that we can write a
  common argument. 

  Let $\Lambda:=\Lambda_\red\sqcup\Lambda_\blue$.
  Condition $\mathcal{R}_1$ we arranged by constructing trees.
  Conditions $\mathcal{R}_2$, and $\mathcal{F}_1$ are immediate.
We check conditions $\mathcal{R}_3$ and $\mathcal{R}_4$.
The only interesting
 cases are squares and cycles that include vertices from different
 components of $\Gamma\setminus\cut{a}{b}$.

 Suppose $x_1\in\Gamma_1$ and $x_2\in\Gamma_2$ are the
 diagonals of a square.
Since $\cut{a}{b}$ is a cut, the square is $\{x_1,x_2\}\join\{a,b\}$.
 Since the only $\Lambda_\blue$ connection between $\Gamma_1$ and
 $\Gamma_2$ is through $c_1$ and $c_2$,  we have:
 \[\hull_\Lambda\{a,b\}\join\hull_\Lambda\{x_1,x_2\}=\{a,b\}\join
   (\hull_{\Lambda_{\Gamma_1,\blue}}\{x_1,c_1\}\cup\hull_{\Lambda_{\Gamma_2,\blue}}\{c_2,x_2\}\cup\hull_{\Lambda_\blue}\{c_1,c_2\})\]
 This join is indeed contained in $\Gamma$, because $x_i,c_i\in\lk_{\bar\Gamma_i}(a)\cap\lk_{\bar\Gamma_i}(b)$, and by
Corollary~\ref{cor:linkintersectionconvex},
$\lk_{\bar\Gamma_i}(a)\cap\lk_{\bar\Gamma_i}(b)$ is
$\Lambda_{\bar\Gamma_i}$--convex, so all of
$\hull_{\Lambda_{\Gamma_i,\blue}}\{x_i,c_i\}$ are common neighbors of $a$
and $b$, whereas $\hull_{\Lambda_\blue}\{c_1,c_2\}\subset
\lk(a)\cap\lk(b)$ by construction. 
Thus, $\mathcal{R}_3$ is satisfied.

If $\mathcal{R}_4$ is violated there is a shortest cycle $\gamma$
containing an edge $e$ that is not included in a square with vertices
in $\hull_\Lambda(\gamma)$.
We may assume such a cycle is induced, since if it has vertices that
are adjacent in $\Gamma$ and not in $\gamma$ we could use that edge to
cut $\gamma$ into two strictly shorter cycles whose union contains all
the edges of $\gamma$.
In particular, $e$ is in one of these shorter cycles, $\gamma'$.
But $\hull_\Lambda\gamma'\subset\hull_\Lambda\gamma$, so $e$ is not
included in a square with vertices in $\hull_\Lambda\gamma'$,
contradicting that $\gamma$ was a shortest counterexample. 

So, consider an induced cycle $\gamma$ that contains, without loss of
generality, vertices from $\Gamma_1$ and $\Gamma_2$. 
Since the cycle crosses the cut, it contains $a$ and $b$.
Since it is induced, either $\gamma\cap\bar\Gamma_1= a\edge c_1\edge b$, or
$c_1\notin\gamma$.
In the first case, both edges $a\edge c_1$ and $c_1\edge b$ are in the
square $\{c_1,c_2\}\join\{a,b\}$, and $c_2\in\hull_\Lambda(\gamma)$ since
the only $\Lambda_\blue$--connection between $\Gamma_1$ and $\Gamma_2$ is through $c_2$.
On the other hand, if $c_1\notin\gamma$ then there is a cycle $\gamma'$
made from
$\gamma\cap\bar\Gamma_1$ and the segment $a\edge c_1\edge b$ that is completely
contained in $\bar\Gamma_1$, so its edges all participate in squares
with vertices in $\hull_{\Lambda_{\Gamma_1}}(\gamma')$, which is
contained in $\hull_{\Lambda_{\Gamma}}(\gamma)$, since the vertex
$c_1=\gamma'\setminus\gamma$ is in the $\Lambda_\blue$--hull of
$\gamma$, being the only $\Lambda_\blue$--connection between
$\bar\Gamma_1$ and $\bar\Gamma_2$.
This shows that all $\bar\Gamma_1$--edges of $\gamma$ are contained in
a square with vertices in $\hull_\Lambda(\gamma)$.
  \end{proof}

\section{Satellite-dismantlability and the Coning Algorithm} \label{SecConing}
In this section we take a graph with a FIDL--$\Lambda$ apart and then put it back together again. 
\begin{thm}\label{satellitedismantling}
      Let $\Gamma$ be an incomplete, triangle-free graph with no separating clique that
      admits a FIDL--$\Lambda$.
      Then $\Gamma$ is satellite-dismantlable to a square through a
      sequence of graphs
      $\Gamma_0=\Gamma\supset\Gamma_1\supset\dots\supset\Gamma_n$ such that for
      all $i$, $\Gamma_i$ is incomplete, triangle-free with no separating clique and has
      $\Lambda\cap\Gamma_i$ as a FIDL--$\Lambda$, with the satellite vertex
      $\Gamma_i\setminus\Gamma_{i+1}$ being a leaf in  $\Lambda\cap\Gamma_i$.
    \end{thm}
\begin{proof}
  Our goal is to find a leaf $v$ of $\Lambda$ such that $v$ is a
  satellite in $\Gamma$ and $\Gamma_1:=\Gamma\setminus\{v\}$ has the
  desired properties.
  The proof is easy if $\Gamma$ is a suspension, so assume not.

  Identify $\Delta$ with its image in $\diag(\Gamma)$.
  This graph is connected, since $W_\Gamma$ and $A_\Delta$
  are one-ended.
  We claim that every vertex of $\diag(\Gamma)$ is adjacent to a
  vertex of $\Delta$.
  Suppose that $\{a,b\}\in\diag(\Gamma)\setminus\Delta$.
  By definition, $\{a,b\}$ is a diagonal of some square, so $a$ and $b$
  have at least two common neighbors.
  By Corollary~\ref{cor:linkintersectionconvex}, $\lk(a)\cap\lk(b)$ is
  $\Lambda$--convex, so it  contains a subtree of $\Lambda$ with at least one edge.
  Any such edge gives a vertex of $\Delta$ adjacent to $\{a,b\}$ in
  $\diag(\Gamma)$.
  Note that this shows $\Gamma$ is strongly \CFS.
  
  The vertex $v$ is a leaf of $\Lambda$ if and only if it has a unique neighbor $v'$ in
  $\Lambda$.
  Equivalently, $\{v,v'\}$ is the unique vertex of $\Delta$
  containing $v$.
  We claim that $v$ is a satellite of $v'$.
  Let $w\in\lk(v)$.
  Since $\{v,v'\}$ is the diagonal of some square,
  $|\lk(v)\cap\lk(v')|\geq 2$, so there exists $c\in(\lk(v)\cap\lk(v'))\setminus\{w\}$.
  Since the edge $v\edge c$ is not a separating clique, there is
  a path from $v'$ to $w$ that does not go through $v$ or $c$, so there is a cycle $\gamma$
  that goes 
  from $v$ to $c$ to $v'$ to $w$ and then back across the edge between
  $v$ and $w$.
  By $\mathcal{R}_4$ there is a square $\{v,x\}\join\{w,u\}$ with
  $x,u\in\hull_\Lambda(\gamma)$.
  But then $\mathcal{R}_3$ implies
  $\{w\}*\hull_\Lambda\{v,x\}\subset\Gamma$.
  Since $v$ was a $\Lambda$--leaf, $v'\in\hull_\Lambda\{v,x\}$, so $w$
  is adjacent to $v'$. Thus, $\lk(v)\subset\lk(v')$.

  The claim that $\Lambda\setminus\{v\}$ is a FIDL--$\Lambda$ is automatic; all of the defining properties are inherited from $\Lambda$.
  The key for this is that since only a leaf of $\Lambda$
  was deleted, convex sets in $\Lambda\setminus\{v\}$ are still
  convex in $\Lambda$.

  We must show it is possible to choose a leaf of $\Lambda$ such that
  $\Gamma\setminus\{v\}$ is not separated by a clique. 
First, let $v$ be a $\Lambda$--leaf and $\{v,v'\}$ the unique vertex
of $\Delta$ containing $v$, and suppose
 that $\{v,v'\}$ is not a cut vertex of $\Delta$.
  Then $\diag(\Gamma)\setminus \{v,v'\}$ consists of
  leaves of $\diag(\Gamma)$ that were connected only to $\{v,v'\}$,
  plus one additional component $\Omega$ containing
  $\Delta\setminus\{v,v'\}$.
  The reason for this is that $\Gamma$ being triangle-free implies
  $\diag(\Gamma)$ is triangle-free, so if $\{a,b\}$ is adjacent to $\{v,v'\}$,
  but is not a leaf of $\diag(\Gamma)$, then it is also adjacent to
  some other vertex $\{c,d\}$, which is adjacent to a vertex of
  $\Delta$, but not to $\{v,v'\}$.

  We further note that $\Omega$ has support $\Gamma\setminus\{v\}$.
  Since $\Gamma$ is not a suspension, 
  $v'$ is not a $\Lambda$--leaf, so it is contained in at least one other 
  vertex of $\Delta$, so
  $v'\in\supp(\Omega)$.
  For other vertices, we only worry about those appearing in 
  a leaf  $\{a,b\}$  of $\diag(\Gamma)$
  connected to $\{v,v'\}$.
  Since $\{v,v'\}$ is not a cut vertex
  of $\Delta$, $\{a,b\}\notin \Delta$.
  But $\Lambda$ spans $\Gamma$, so there is some vertex of $\Delta$
  with $a$ in its support, and likewise for $b$.
  
  Now consider $\diag(\Gamma\setminus\{v\})$.
  If $\{a,b\}$ is a leaf of $\diag(\Gamma)$ connected only to
  $\{v,v'\}$ then the only square of $\Gamma$ with $\{a,b\}$ as a
  diagonal is $\{a,b\}\join\{v,v'\}$, so $\{a,b\}$ is not the diagonal
  of a square in $\Gamma\setminus\{v\}$ and does not appear as a
  vertex of $\diag(\Gamma\setminus\{v\})$.
  Since $\{v,v'\}$ was the unique vertex of $\diag(\Gamma)$ containing
  $v$, $\diag(\Gamma\setminus\{v\})=\Omega$, which is connected, by hypothesis, and
  has full support in $\Gamma\setminus\{v\}$.
  Thus, $\Gamma\setminus\{v\}$ is strongly CFS, and Lemma~\ref{cfsimpliesoneended} says $\Gamma\setminus\{v\}$ has no
  separating clique. 

  Now we argue that there always exists a $\Lambda$--leaf not giving a
  cut vertex of $\Delta$, so we can find it and apply the previous argument.
Suppose the first chosen $\Lambda$--leaf $v$ does give a cut vertex
$\{v,v'\}$ of $\Delta$.
This implies that $\Gamma$ has a cut $\cut{v}{v'}$, since $W_\Gamma$ and
$A_\Delta$ split over a two-ended subgroup commensurable to  $\langle vv'\rangle$.
Let $v_0:=v$ and $v_0':=v'$.
By Corollary~\ref{lem:componentsconvex}, for each component $\Gamma'$
of $\Gamma\setminus\cut{v_0}{v_0'}$, the intersection of each component of
$\Lambda$ with $\bar\Gamma'$ is $\Lambda$--convex, so $\Gamma'$
contains at least one $\Lambda$--leaf $v_1$.
If $v_1$ uniquely appears as
$\{v_1,v_1'\}$ in $\Delta$ and $\{v_1,v_1'\}$ is not a cut vertex of
$\Delta$, we are done.
Otherwise, $\cut{v_1}{v_1'}$ is a cut of $\Gamma$.
Since $\Gamma$ has no crossing cuts,
$\cut{v_0}{v_0'}\setminus\cut{v_1}{v_1'}$ is a non-empty set contained in a
single component of $\Gamma\setminus\cut{v_1}{v_1'}$.
Choose a component $\Gamma''$ of
$\Gamma\setminus\cut{v_1}{v_1'}$ not containing $\cut{v_0}{v_0'}\setminus\cut{v_1}{v_1'}$ and repeat, always choosing a
complementary component of the most recent cut that does not contain the previous
cuts, so that the size of the components strictly decreases at each
step. 
If $v_{i+1}$ is the lone vertex in its component of
$\Gamma\setminus\cut{v_i}{v_i'}$ then it cannot be part of a cut, since
the cut would cross $\cut{v_i}{v_i'}$, so eventually this process
produces a $\Lambda$--leaf that does not appear in a cut vertex of
$\Delta$. 
\end{proof}
\begin{cor}
  If $\Gamma$ is an incomplete, triangle-free graph with no separating
  clique and no satellite vertex then $\Gamma$ does not admit a FIDL--$\Lambda$.
\end{cor}
\begin{cor}\label{cor:strongcfs}
  If $\Gamma$ is an incomplete, triangle-free graph with no separating
  clique and it admits a FIDL--$\Lambda$ then it is strongly \CFS.
\end{cor}

\begin{alg} \label{ConingAlg}  \leavevmode
  We perform the following inductive procedure to build a graph $\Gamma$ with an associated graph $\Lambda\leq \Gamma^c$ with two connected components:
    \begin{enumerate}[topsep=2pt]
        \item The initial graph $\Gamma_0$ is a square, the associated graph $\Lambda_0 = \Gamma_0^c$ is the complement graph of $\Gamma_0$.
        \item Build a sequence of pairs $((\Gamma_0,\Lambda_0), (\Gamma_1,\Lambda_1), \dots, (\Gamma_n,\Lambda_n))$ by applying the induction step:
       Given $(\Gamma_i,\Lambda_i)$, pick a vertex $v_{i}$ and a set $N_{i} \subseteq \lk_{\Gamma_i}(v_i)$ of at least two vertices that is $\Lambda_i$--convex.
          Define $\Gamma_{i+1}$ by coning-off $N_i$ to $x_{i+1}$, and define $\Lambda_{i+1}$ by adding an edge from $v_i$ to $x_{i+1}$ to $\Lambda_i$. 
    \end{enumerate}
  \end{alg}
Note that at each step, $x_{i+1}$ is a satellite of $v_i$ in $\Gamma_{i+1}$.

\begin{thm} \label{thmConingDecompOneWay}
 If the pair $(\Gamma,\Lambda)$ can be constructed by the Coning Algorithm \ref{ConingAlg}, then $\Gamma$ is incomplete, triangle-free and has no separating cliques, and $\Lambda$ is a FIDL--$\Lambda$ for $\Gamma$.
\end{thm}
\begin{proof}
  The proof is by induction on $n$.
  $\Theta(\Gamma_0,\Lambda_0)$ satisfies all the conditions.
  Assume $\Theta(\Gamma_i,\Lambda_i)$ does too. 
  We show that $\Theta(\Gamma_{i+1},\Lambda_{i+1})$ satisfies
  $\mathcal{R}_3$ and $\mathcal{R}_4$. The other conditions are easy
  to verify.

    Since $\Theta(\Gamma_i,\Lambda_i)$ satisfies condition $\mathcal{R}_3$, we only need to check squares in $\Gamma_{i+1}$ containing the new vertex $x_{i+1}$.
By construction of $\Lambda_{i+1}$, any such square is of the form $\{x_{i+1},l\}\join\{n,n^\prime\}$ with $n,n^\prime \in N_i$, and, since $N_i=\lk_{\Gamma_{i+1}}(x_{i+1})$ is $\Lambda_i$--convex, $\{x_{i+1}\}\join \hull_{\Lambda_{i+1}}\{n,n^\prime\}\subset\Gamma_{i+1}$.
    
    The vertex $x_{i+1}$ is only connected to $v_i$ in $\Lambda_{i+1}$.
    Hence, $\hull_{\Lambda_{i+1}}\{x_{i+1},l\}\setminus \{x_{i+1}\} = \hull_{\Lambda_{i+1}}\{v_i,l\}=\hull_{\Lambda_{i}}\{v_i,l\}$.
    Since $\Gamma_i$ is triangle-free, $v_i$ and $l$ are not adjacent, so 
    $\{v_i,l\}\join\{n,n^\prime\} \subseteq \Gamma_{i}$ is a square.
    Condition $\mathcal{R}_3$ for $\Theta(\Gamma_i,\Lambda_i)$ implies $\hull_{\Lambda_i}\{v_i,l\}\join\hull_{\Lambda_i}\{n,n^\prime\}\subset\Gamma_i$.
    Condition $\mathcal{R}_3$ is satisfied, since:
    \[\hull_{\Lambda_{i+1}}\{x_{i+1},l\}\join\hull_{\Lambda_{i+1}}\{n,n^\prime\}=\{x_{i+1}\}\join\hull_{\Lambda_{i}}\{n,n^\prime\}\cup
      \hull_{\Lambda_i}\{v_i,l\}\join\hull_{\Lambda_i}\{n,n^\prime\}\subset\Gamma_{i+1}\]

    Since $\Theta(\Gamma_i,\Lambda_i)$ satisfies condition $\mathcal{R}_4$, we only need to check cycles containing $x_{i+1}$.
    Let $\gamma$ be a cycle containing $x_{i+1}$.
    Since $\lk_{\Gamma_{i+1}}(x_{i+1})=N_i$, $\gamma$ is of the form $\gamma = (x_{i+1},n,l_1,\dots, l_k,n^\prime)$ for $n,n^\prime \in N_i$.
The edges incident to $x_{i+1}$ are contained in the square
$\{x_{i+1},v_i\}\join\{n,n'\}$, all of whose vertices are in $\gamma$,
except possibly $v_i$.
    But $x_{i+1}$ is a leaf of $\Lambda_{i+1}$ connected only to $v_i$, so $v_i$ is certainly in $\hull_{\Lambda_{i+1}}(\gamma)$.
    For the remaining edges of $\gamma$, replace $\gamma$ by $\gamma^\prime = (v_i,n,l_1,\dots, l_k,n^\prime)$, which is a loop in $\Gamma_i$ that can be split into at most two cycles in $\Gamma_i$. Condition $\mathcal{R}_4$ for $\Theta(\Gamma_i,\Lambda_i)$ implies each of these edges belongs to a square with vertices in the $\Lambda_i$--hull of its cycle, which is a subset of the $\Lambda_{i+1}$--hull of $\gamma$.
    \end{proof}

\begin{exam} \label{ExampleConingDecompCubeWithDiag}
    The graph $\Gamma$ of Example~\ref{ExampleCubeWithDiagonal}  that is the 1--skeleton of a 3-cube with one
    space diagonal has a FIDL-$\Lambda$.
    The pair $(\Gamma,\Lambda)$ can be constructed by the Coning
    Algorithm \ref{ConingAlg}. A coning sequence $(x_{i},\Gamma_i,\Lambda_i, v_i,N_i)$  is illustrated in Figure \ref{fig:ConingDecompCubeWithDiag}.
    \begin{figure}[ht!]
        \centering
        \scriptsize
        \def\svgwidth{.95\textwidth}
        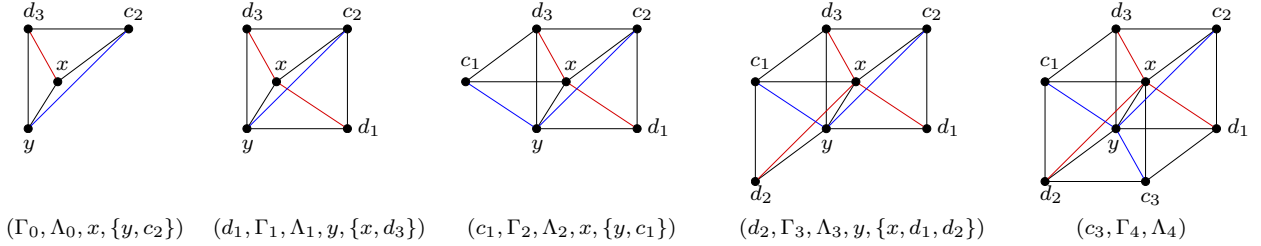
        \caption{Example of a Coning Sequence}
        \label{fig:ConingDecompCubeWithDiag}
    \end{figure}
\end{exam}

\begin{thm}\label{maintheorem}
  Let $\Gamma$ be an incomplete, triangle-free graph without separating cliques.
  $\Gamma$ admits a FIDL--$\Lambda$ if and only if it admits a
  satellite-dismantling sequence
  $\Gamma=\Gamma_n\supset\Gamma_{n-1}\supset\dots\supset \Gamma_0$
  such that $\Gamma_0$ is a square, each $\Gamma_i$ is incomplete and
  triangle-free with no separating clique, and the following condition
  is satisfied:
For $0\leq i<n$ let $\{x_{i+1}\}:=\Gamma_{i+1}\setminus\Gamma_i$ and
$N_i:=\lk_{\Gamma_{i+1}}(x_{i+1})$. 
\begin{equation}
  \label{eq:1}
\forall  i<n,\,  \exists v_i\in V_i:=\{v\in\Gamma_i\mid N_i\subset\lk_{\Gamma_i}(v)\},\,\forall j>i,\, \text{if }x_{i+1}\in
N_j\text{ and }N_j\cap\Gamma_i\neq\emptyset\text{ then } v_i\in N_j
  \tag{$\dagger$}
\end{equation}
\end{thm}
\begin{proof}
 Suppose we have a satellite-dismantling sequence for
$\Gamma$ satisfying the given conditions. 
Let $\Lambda_0:=\Gamma_0^c$, and construct $\Lambda_{i+1}$ from
$\Lambda_i$ by choosing $v_i$ satisfying \eqref{eq:1}, and adding a
$\Lambda$--edge from $v_i$ to $x_{i+1}$. 
To apply Theorem~\ref{thmConingDecompOneWay} we must verify that $N_j$
is $\Lambda_j$--convex, for each $j$.
Given $j$, let $k_j\geq 0$ be the minimum index such that
$N_j\cap\Gamma_{k_j}\neq\emptyset$.
The vertices of $N_j$ are pairwise at distance 2 from one another, so
they all have the same color; assume it is $\blue$. 
If $y\in N_j$ then either 
$y\in\Gamma_{k_j}$ or $y=x_{i+1}$ for some $k_j\leq i<j$.
In the second case, by \eqref{eq:1}, $y$ is $\Lambda_j$--adjacent to
$v_i\in N_j$, which is in a lower stratum.
In this way, \eqref{eq:1} implies that every vertex in $N_j$ can be
$\Lambda_j$--connected through vertices in $N_j$ to a vertex in $N_j\cap \Gamma_{k_j}$.
If $k_j>0$ then $\Gamma_{k_j}\setminus\Gamma_{k_j-1}$ is a single
vertex, and if $k_j=0$ then the two vertices $\Lambda_{0,\blue}$ are
connected by an edge.
Thus, $N_j$ is $\Lambda_j$--convex. 

  Conversely, 
  if $\Gamma$ admits a FIDL--$\Lambda$ then, by
  Theorem~\ref{satellitedismantling}, it admits a
  satellite-dismantling sequence leading to a square through graphs
  with the desired properties.
  For each satellite 
  $x_{i+1}$ there is a unique $\Lambda_{i+1}$--edge at $x_{i+1}$ connecting it
  to a $v_i\in V_i$.
  We check that this $v_i$ satisfies \eqref{eq:1}.

  Suppose for some $j$ there is an $i$ such that  $x_{i+1}\in N_j$ and
  there exists $y\in\Gamma_i\cap N_j$. Assume
  $N_j\subset\Lambda_\red$.
Lemma~\ref{lem:linksconvex} says 
$N_j:=\lk_{\Gamma_{j+1}}(x_{j+1})$ is $\Lambda_{j+1}$-convex, but
$\hull_{\Lambda_j}(N_j)=\hull_{\Lambda_{j+1}}(N_j)$, so $N_j$ is 
$\Lambda_j$--convex.
Now we have two points $x_{i+1}$ and $y$ that are contained in two
different subtrees of $\Lambda_{j,\red}$: one is 
$N_j\cap\Lambda_{j,\red}$ and the other is $\Lambda_{i+1,\red}$.
Thus, the $\Lambda_{j,\red}$--geodesic from $x_{i+1}$ to $y$, the
$N_j\cap\Lambda_{j,\red}$--geodesic from $x_{i+1}$ to $y$ and the  
$\Lambda_{i+1,\red}$--geodesic from $x_{i+1}$ to $y$ coincide.
We know that the $\Lambda_{i+1,\red}$--geodesic from $x_{i+1}$ to $y$ goes through $v_i$, since $x_{i+1}$ is a leaf of $\Lambda_{i+1,\red}$
connected only to $v_i$.
Thus $v_i\in N_j$, and \eqref{eq:1} is satisfied. 
\end{proof}

\begin{prop}\label{mixed_multiple_trees_are_DL}
  Let $T$ be a finite tree, with at least one edge, whose vertices are labelled by natural
  numbers $n(v)$ such that 1 occurs only on leaves and if $T$ is a
  single edge then both labels are greater than 1.
Let  $\Gamma$ be the graph that has, for each $v\in T$, $n(v)$--many vertices
$(v,0),\dots,(v,n(v)-1)$.
Connect $(v,i)$ to $(w,j)$ by an edge if and only if $v$ is adjacent
to $w$ in $T$.
Then $\Gamma$ admits a FIDL--$\Lambda$. 
\end{prop}
\begin{proof}
We describe a satellite-dismantling and then say why it satisfies
  Theorem~\ref{maintheorem}.

  If $v\in T$ with $n(v)>1$ then for $i+1\leq n(v)-1$, $(v,i+1)$ is a satellite of $(v,i)$,
  coning off $\{(w,j)\mid w\in\lk_T(v),\,0\leq
  j<n(w)\}$.
  Pick any such $v$ and iteratively remove $(v,n(v)-1)$, \dots,
  $(v,2)$.
  Repeat for all $v$ with $n(v)>2$.
  This reduces us to the case that $n$ is bounded by 2. 
 
 If $\Gamma$ is not a square and there is a leaf $v$ of $T$ attached
 to $w$ with $n(v)=1$ then there exists $x\in T$ at distance 2 from $v$.
 In $\Gamma$, $(v,0)$ is a satellite of $(x,0)$ coning-off
 $\{(w,0),(w,1)\}$.
 In this way, remove all $T$--leaves with $n$--value 1, until either
 we reach a square or  have smaller $\Gamma$ and $T$ with $n\equiv 2$.

 Suppose $n\equiv 2$  and $\Gamma$ is not a square, so $T$
 is not a single edge. 
If $v$ is a leaf of $T$ attached to $w$ then $(v,1)$ is a satellite
of $(v,0)$ coning-off
 $\{(w,0),(w,1)\}$.
 Thus, we reduce $n$ to 1 on each of the leaves of $T$, and then rerun the
 case in the previous paragraph.

Iterating gives a satellite-dismantling reducing $\Gamma$ to a
square. Furthermore, in the dismantling we specified for each
satellite which vertex we were considering it as the satellite of,
which allows an explicit construction of $\Lambda$.
It is observed that the coned-off set is $\Lambda$--convex at each
step, so this is a FIDL.
\end{proof}
\begin{remark}
  The proof relies on the fact that $T$ is a tree because after
  reducing to $n\leq 2$ the dismantling proceeds by removing
  $T$--leaves. Proposition~\ref{mixed_multiple_trees_are_DL} cannot be true for
arbitrary graphs $\Delta$, as if $\Delta$ has a long isometrically
embedded loop then so does $\Gamma$, which prevents the existence of a
FIDL--$\Lambda$, by Proposition~\ref{nolongcycles}.
However, if $\Delta$ is an arbitrary connected graph and $n\equiv 2$
then the $\Gamma$ constructed from $\Delta$ as above is commensurable
to a RAAG, by \cite{davisjanus}.
\end{remark}

\begin{exam}\label{ex:mixed_multiple_trees_are_DL}
  Figure~\ref{fig:mixedmultipletree} works out an example of a tree as
  in Proposition~\ref{mixed_multiple_trees_are_DL} in detail.
  \begin{figure}[h!]
    \centering
    \tikz[baseline=-9ex]{\tiny
  \coordinate[label={270:1}, label={90:$a$}] (0) at (0,0);
   \coordinate[label={270:4}, label={90:$b$}] (1) at (1,0);
    \coordinate[label={270:3}, label={ 90:$c$}] (2) at (2,0);
     \coordinate[label={270:2}, label={90:$d$}] (3) at (3,0);
     \filldraw (0) circle (1pt) (1) circle (1pt) (2) circle (1pt) (3)
     circle (1pt);
     \draw (0)--(1)--(2)--(3)}
   \hfill
   \begin{tikzpicture}[baseline=-6ex, scale=2]\tiny
\coordinate[label={[label distance=-4pt] 225:$b_0$}] (0) at (-0.5,0.25);
\coordinate[label={[label distance=-1pt] 315:$c_0$}] (1) at (0,0.25);
\coordinate[label={[label distance=-2pt] 315:$c_1$}] (2) at (0,0.125);
\coordinate[label={[label distance=-2pt] 315:$c_2$}] (3) at (0,0);
\coordinate[label={[label distance=-1pt] 180:$a_0$}] (4) at (-1,0.25);
\coordinate[label={[label distance=-4pt] 225:$b_1$}] (5) at (-0.5,0.125);
\coordinate[label={[label distance=-4pt] 225:$b_2$}] (6) at (-0.5,0);
\coordinate[label={[label distance=-4pt] 225:$b_3$}] (7) at (-0.5,-.125);
\coordinate[label={[label distance=-1pt] 0:$d_0$}] (8) at (0.5,0.25);
\coordinate[label={[label distance=-1pt] 0:$d_1$}] (9) at (0.5,0.125);
\draw (0)--(1);
\draw (0)--(2);
\draw (0)--(3);
\draw (0)--(4);
\draw (1)--(5);
\draw (1)--(6);
\draw (1)--(7);
\draw (1)--(8);
\draw (1)--(9);
\draw (2)--(5);
\draw (2)--(6);
\draw (2)--(7);
\draw (2)--(8);
\draw (2)--(9);
\draw (3)--(5);
\draw (3)--(6);
\draw (3)--(7);
\draw (3)--(8);
\draw (3)--(9);
\draw (4)--(5);
\draw (4)--(6);
\draw (4)--(7);
\filldraw[red] (0) circle (.5pt);
\filldraw[blue] (1) circle (.5pt);
\filldraw[blue] (2) circle (.5pt);
\filldraw[blue] (3) circle (.5pt);
\filldraw[blue] (4) circle (.5pt);
\filldraw[red] (5) circle (.5pt);
\filldraw[red] (6) circle (.5pt);
\filldraw[red] (7) circle (.5pt);
\filldraw[red] (8) circle (.5pt);
\filldraw[red] (9) circle (.5pt);
\draw[color=blue] (1)--(2)--(3) (1) .. controls (0,.5) and (-1,.5) .. (4);
\draw[color=red] (0)--(5)--(6)--(7) (8)--(9) (0) .. controls (-.5,.5) and (.5
,.5) .. (8);
\end{tikzpicture}
\hfill
\begin{tikzpicture}[scale=2]\tiny
\coordinate[label={[label distance=-1pt] 270:$b_0b_1$}] (0) at (-1,0);
\coordinate[label={[label distance=0pt] 270:$c_0c_1$}] (1) at (.5,.4);
\coordinate[label={[label distance=-2pt] 270:$b_0b_2$}] (2) at (-1,.2);
\coordinate[label={[label distance=-1pt] 270:$b_0b_3$}] (3) at (-1,.4);
\coordinate[label={[label distance=-1pt] 0:$b_0d_0$}] (4) at (2,0);
\coordinate[label={[label distance=-1pt] 0:$b_0d_1$}] (5) at (2,.1);
\coordinate[label={[label distance=-1pt] 90:$c_0c_2$}] (7) at (.5,.6);
\coordinate[label={[label distance=-1pt] 180:$a_0c_0$}] (8) at (-2,.4);
\coordinate[label={[label distance=-1pt] 90:$c_1c_2$}] (9) at (.5,.8);
\coordinate[label={[label distance=-1pt] 180:$a_0c_1$}] (10) at (-2,.5);
\coordinate[label={[label distance=-1pt] 180:$a_0c_2$}] (11) at (-2,.6);
\coordinate[label={[label distance=-1pt] 90:$b_1b_2$}] (12) at (-1,.6);
\coordinate[label={[label distance=-2pt] 90:$b_1b_3$}] (13) at (-1,.8);
\coordinate[label={[label distance=-1pt] 0:$b_1d_0$}] (14) at (2,.2);
\coordinate[label={[label distance=-1pt] 0:$b_1d_1$}] (15) at (2,.3);
\coordinate[label={[label distance=-1pt] 90:$b_2b_3$}] (17) at (-1,1);
\coordinate[label={[label distance=-1pt] 0:$b_2d_0$}] (18) at (2,.4);
\coordinate[label={[label distance=-1pt] 0:$b_2d_1$}] (19) at (2,.5);
\coordinate[label={[label distance=-1pt] 0:$b_3d_0$}] (21) at (2,.6);
\coordinate[label={[label distance=-1pt] 0:$b_3d_1$}] (22) at (2,.7);
\coordinate[label={[label distance=-1pt] 0:$d_0d_1$}] (24) at (2,1);

\draw (0)--(9);
\draw (0)--(10);
\draw (0)--(11);
\draw (1)--(13);
\draw (1)--(14);
\draw (1)--(15);
\draw (1)--(18);
\draw (1)--(19);
\draw (1)--(21);
\draw (1)--(22);
\draw (1)--(24);
\draw (2)--(9);
\draw (2)--(10);
\draw (2)--(11);
\draw (3)--(9);
\draw (3)--(10);
\draw (3)--(11);
\draw (4)--(7);
\draw (5)--(9);
\draw (7)--(12);
\draw (7)--(13);
\draw (7)--(14);
\draw (7)--(15);
\draw (7)--(17);
\draw (7)--(18);
\draw (7)--(19);
\draw (7)--(21);
\draw (7)--(22);
\draw (7)--(24);
\draw (8)--(13);
\draw (9)--(12);
\draw (9)--(13);
\draw (9)--(14);
\draw (9)--(15);
\draw (9)--(17);
\draw (9)--(18);
\draw (9)--(19);
\draw (9)--(21);
\draw (9)--(22);
\draw (9)--(24);
\draw (10)--(12);
\draw (10)--(13);
\draw (10)--(17);
\draw (11)--(12);
\draw (11)--(13);
\draw (11)--(17);
\draw (0)--(7)  (3)--(7) (7)--(5) (7)--(2)  (2)--(1) (2)--(8) (1)--(4) 
(3)--(1) (3)--(8) (1)--(5); 
\draw (9)--(22) (22)--(1);
\draw[thick, purple] (0)--(1) (0)--(8)
(24)--(1) (24)--(9) (9)--(0) (9)--(12) (9)--(17) (1)--(12)--(8)
(1)--(17)--(8) (4)--(9) (4)--(1);
\filldraw[purple] (0) circle (.5pt);
\filldraw[purple] (1) circle (.5pt);
\filldraw (2) circle (.5pt);
\filldraw (3) circle (.5pt);
\filldraw (22) circle (.5pt);
\filldraw (5) circle (.5pt);
\filldraw (7) circle (.5pt);
\filldraw[purple] (8) circle (.5pt);
\filldraw[purple] (24) circle (.5pt);
\filldraw[purple] (9) circle (.5pt);
\filldraw (10) circle (.5pt);
\filldraw (11) circle (.5pt);
\filldraw[purple] (12) circle (.5pt);
\filldraw (13) circle (.5pt);
\filldraw (14) circle (.5pt);
\filldraw (15) circle (.5pt);
\filldraw[purple] (17) circle (.5pt);
\filldraw (18) circle (.5pt);
\filldraw (19) circle (.5pt);
\filldraw (21) circle (.5pt);
\filldraw[purple] (4) circle (.5pt);
\end{tikzpicture}
    \caption{An example $T$, $\Theta(\Gamma,\Lambda)$, and $\Delta$
      (purple) sitting in $\diag(\Gamma)$.}
    \label{fig:mixedmultipletree}
  \end{figure}
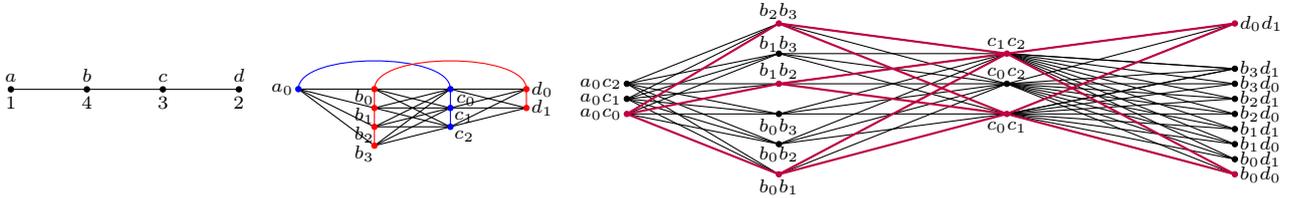
\end{exam}

\section{A search algorithm}\label{search}

\begin{defin}
  A \emph{FIDL--$\Lambda$ for $\Gamma$ relative to} a
  collection of pairs of vertices $\{\{p_0,q_0\},\dots,\{p_n,q_n\}\}$
  of $\Gamma$ is a FIDL--$\Lambda$
  that contains each of the pairs $\{p_i,q_i\}$ as an edge of $\Lambda$.
\end{defin}
\begin{ralg}\label{relativesearch}
   Given an incomplete, triangle-free graph $\Gamma$ without
   separating cliques, find a  FIDL--$\Lambda$ relative to
   $\{\{p_0,q_0\},\dots,\{p_n,q_n\}\}$ or decide that one does not
   exist as follows:
   \begin{enumerate}
     \item If $\Gamma$ is not strongly \CFS\, stop; No FIDL--$\Lambda$
       exists. \label{item:strongcfs}
       \item If $\Gamma$ contains a cycle violating Proposition~\ref{nolongcycles},
         stop; No FIDL--$\Lambda$
       exists. \label{item:nolongchordlesscycles}
   \item If there exists $x_0,
     x_1,\dots,x_{n-1}$ such that for each $i<n$ there exists $j$ with
     $\{x_i,x_{i+1}\}=\{p_j,q_j\}$, then stop; No relative FIDL--$\Lambda$
     exists.\label{item:cycle}
       \item Enumerate satellite-dismantling sequences reducing
         $\Gamma$ to a square through graphs without separating
         cliques.
         If none exist, stop; No FIDL--$\Lambda$ exists.\label{item:dismantle}
    \item For each such satellite-dismantling sequence, check, in
      the notation of Theorem~\ref{maintheorem}, if Condition
      \eqref{eq:1}  is satisfied.
      Moreover, if we assume for each $k$ that for all $i$,
      $q_k\in\Gamma_i\implies p_k\in\Gamma_i$, then we require for
      each $k$ that either $\{p_k,q_k\}$ is a diagonal of the square
      $\Gamma_0$ or that for $i$ such that $q_k=x_{i+1}$, we have that
      $v_i=p_k$ is
      a choice for $v_i\in V_i$ satisfying Condition
      \eqref{eq:1}. \label{item:relcheck}
      \item If a suitable satellite-dismantling sequence is found,
        then Theorem~\ref{maintheorem} provides the relative
        FIDL--$\Lambda$. Otherwise, no relative FIDL--$\Lambda$ exists.
   \end{enumerate}
\end{ralg}
\begin{proof}
  Item~\eqref{item:strongcfs} is Corollary~\ref{cor:strongcfs}.
  If the given $\{\{p_0,q_0\},\dots,\{p_n,q_n\}\}$ contains a cycle as
  described in \eqref{item:cycle}, then the resulting $\Lambda$
  contains a cycle, violating $\mathcal{R}_1$. Item
  \eqref{item:dismantle} is Theorem~\ref{satellitedismantling}.
  Item \eqref{item:relcheck} describes how it is possible to achieve
  the $\{p_k,q_k\}$ as edges of FIDL--$\Lambda$ from the argument of Theorem~\ref{maintheorem}.
\end{proof}

\begin{galg} \label{SearchAlg}
  Given an incomplete, triangle-free graph $\Gamma$ without separating
  cliques, find a FIDL--$\Lambda$ or decide that one does not exist as follows:
  \begin{enumerate}
    \item If $\Gamma$ is not strongly \CFS, stop; No FIDL--$\Lambda$
      exists. \label{item:cfsagain}
    \item If $\Gamma$ contains a cycle violating Proposition~\ref{nolongcycles},
         stop; No FIDL--$\Lambda$
       exists. \label{item:nolongchordlesscyclesagain}
    \item Compute the JSJ graph of cylinders for $W_\Gamma$ in terms
      of $\Gamma$ as described in
      \cite[Theorem~3.29]{edletzberger2021quasi} (recall Theorem~\ref{ThmJSJGOCforRACGs}).
      If it has a hanging vertex, stop; $W_\Gamma$ is not
      even quasiisometric to a RAAG.\label{item:jsj} 
    \item For each subgraph $\Gamma'$ corresponding to a rigid vertex
      of the graph of cylinders, use Relative Search Algorithm~\ref{relativesearch} to find a
      FIDL--$\Lambda$ for $\Gamma'$ relative to the pairs
      $\{p,q\}\subset\Gamma'$ such that $\Gamma$ has a cut
      $\cut{p}{q}$.
      If this search fails for any rigid vertex, stop; No
      FIDL--$\Lambda$ exists.\label{item:relfidl}
     \item If all rigid vertices have relative FIDL--$\Lambda$ then
       they can be assembled into a FIDL--$\Lambda$ of $\Gamma$.\label{item:assembly}
  \end{enumerate}
\end{galg}
\begin{proof}
  \cite[Proposition~4.43]{thesis}, shows that
  if $\Gamma'$ is the subgraph corresponding to a rigid vertex, then
  it has no separating cliques and is \CFS.
  Iterating the cutting direction of Proposition~\ref{prop:splitting} implies that if $\Gamma$ has a
  FIDL--$\Lambda$ then its restriction to $\Gamma'$ is a
  FIDL--$\Lambda$ for $\Gamma'$ relative to the
  cuts. Together with  Corollary~\ref{cor:strongcfs}, this explains
  the necessity of Items~\eqref{item:jsj} and \eqref{item:relfidl}.

  Torsion-generated groups do not surject $\mathbb{Z}$, so the
  underlying graph of the graph of cylinders is a tree.
Iterating the assembly direction of Proposition~\ref{prop:splitting}
over the cuts combines the relative FIDL--$\Lambda$'s of the subgraphs
of the rigid vertices into a FIDL--$\Lambda$ for $\Gamma$,
establishing Item~\eqref{item:assembly}.
\end{proof}

\section{Performance}\label{comparison}
We have not analyzed the complexity of the naive `enumerate and check
$\Lambda$' vs satellite-dismantling algorithms.
We have an older implementation of the naive algorithm and of the
one in this paper,  and observe that the
satellite-dismantling algorithm performs faster on batches of smallish
($\leq 12$ vertices) graphs.
However, this is not a fair comparison, as the new algorithm also
incorporates other results from this paper, such as that the presence of long
embedded cycle obstructs the existence of a FIDL--$\Lambda$, and that
if $W_\Gamma$ has a non-trivial JSJ decomposition then it suffices to
patch together FIDL--$\Lambda$ for the rigid components.
These could have also been used to improve the naive algorithm. 
Nevertheless, we conjecture that the satellite-dismantling algorithm 
has better generic-case complexity than the naive algorithm, because
it can fail fast: most graphs do not admit a FIDL--$\Lambda$ and the naive algorithm must enumerate and test all $\Lambda$
to confirm none work, but the 
satellite-dismantling algorithm can quit immediately if the graph has
no satellite-dismantling sequence.

One might also wonder about the efficiency of the search for a
suitable satellite-dismantling sequence.
For example, if $\Gamma$ is a triangle-free strongly \CFS\ graph that
has a satellite-dismantling sequence to a square through strongly
\CFS\ graphs, and if $v$ is some satellite vertex of $\Gamma$ such that
$\Gamma\setminus\{v\}$ is strongly \CFS\, does $\Gamma\setminus\{v\}$
admit a satellite-dismantling sequence to a square through strongly
\CFS\ graphs?
Ie, does finding a full sequence depend on choosing satellites in the right order?
If the order does not matter this would speed up the search for satellite-dismantling
sequences: upon the first failure to extend a dismantling sequence all
the way to a square we could immediately quit rather than backtracking
to try different dismantling sequences.
However, once we know that dismantling sequences exist the
story is different; satisfying  Condition~\eqref{eq:1} of Theorem~\ref{maintheorem}
does depend on the chosen sequence, as can be seen in the
following example:
\begin{exam}\label{ex:ordermatters}
  \begin{figure}[h]
    \centering
    $\Gamma=$
  \begin{tikzpicture}[baseline=(0.base)]
    \tiny
    \coordinate[label={[label distance=-1pt] 45:$0$}] (0) at (0,0);
    \coordinate[label={[label distance=-1pt] 90:$1$}] (1) at (0,1);
    \coordinate[label={[label distance=-1pt] 180:$2$}] (2) at (-1,0);
    \coordinate[label={[label distance=-1pt] 270:$3$}] (3) at (0,-1);
    \coordinate[label={[label distance=-1pt] 0:$4$}] (4) at (1,0);
    \coordinate[label={[label distance=-1pt] 90:$5$}] (5) at (-.5,0);
    \coordinate[label={[label distance=-1pt] 90:$6$}] (6) at (0.5,0);
    \draw (0)--(1)--(2)--(3)--(4)--(1) (0)--(3) (2)--(5)--(0)--(6)--(4);
    \filldraw (0) circle (1pt) (1) circle (1pt) (2) circle (1pt) (3)
    circle (1pt) (4) circle (1pt) (5) circle (1pt) (6) circle (1pt);
  \end{tikzpicture}
    \caption{Graph for Example~\ref{ex:ordermatters}.}
    \label{fig:ordermatters}
  \end{figure}
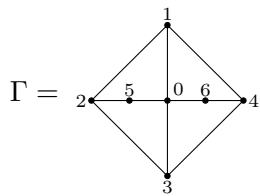
  Consider the graph in Figure~\ref{fig:ordermatters}.
The vertex sequence 6, 5, 0 yields a satellite-dismantling sequence
reducing $\Gamma$  to a square $\Gamma_0=\{1,3\}\join\{2,4\}$,
  such that all $\Gamma_i$  are strongly
  \CFS, and:
  \begin{tabular}{c c c}
    $x_1=0$&$N_0=\{1,3\}$&$V_0=\{2,4\}$\\
    $x_2=5$&$N_1=\{0,2\}$&$V_1=\{1,3\}$\\
    $x_3=6$&$N_2=\{0,4\}$&$V_2=\{1,3\}$\\
  \end{tabular}

  For $i=0$, $x_{i+1}=0\in N_1\cap N_2$, both
  of which intersect $\Gamma_i$, so
  to satisfy Condition~\eqref{eq:1} of Theorem~\ref{maintheorem} we
  would need to choose $v_0\in V_0\cap N_1\cap N_2=\emptyset$.
  If we reconstruct $\Gamma$ by a coning sequence from this data we
  see that there are two possibilities up to symmetry, shown in Figures~\ref{sf:one} and \ref{sf:two}, and
  both contain a vertex with non-$\Lambda$--convex link.

  Consider instead the vertex  sequence 6, 5, 4, which yields a
  satellite-dismantling sequence reducing $\Gamma$ to the square
  $\Gamma_0=\{0,2\}\join\{1,3\}$ with all 
 $\Gamma_i$ 
  strongly \CFS, and:  
  \begin{tabular}{c c c}
    $x_1=4$&$N_0=\{1,3\}$&$V_0=\{0,2\}$\\
    $x_2=5$&$N_1=\{0,2\}$&$V_1=\{1,3\}$\\
    $x_3=6$&$N_2=\{0,4\}$&$V_2=\{1,3\}$\\
  \end{tabular}

  For $i=0$, $x_{i+1}=4\in N_2$ and $N_2\cap\Gamma_i\neq\emptyset$,
  but $x_{i+1}\notin N_1$, so we satisfy
 \eqref{eq:1} by choosing $v_0\in V_0\cap
  N_2=\{0\}$.
  For $i=1,2$,  \eqref{eq:1} imposes no
  condition except $v_i\in V_i$.
  The two possibilities, up to symmetry, shown in Figures~\ref{sf:three} and \ref{sf:four}, are both FIDL--$\Lambda$s.

  \begin{figure}[h]
    \centering
    \begin{subfigure}{.24\textwidth}
          \centering
      \begin{tikzpicture}
    \tiny
    \coordinate[label={[label distance=-1pt] 45:$$}] (0) at (0,0);
    \coordinate[label={[label distance=-1pt] 90:$$}] (1) at (0,1);
    \coordinate[label={[label distance=-1pt] 180:$$}] (2) at (-1,0);
    \coordinate[label={[label distance=-1pt] 270:$$}] (3) at (0,-1);
    \coordinate[label={[label distance=-1pt] 0:$$}] (4) at (1,0);
    \coordinate[label={[label distance=-1pt] 90:$$}] (5) at (-.5,0);
    \coordinate[label={[label distance=-1pt] 90:$$}] (6) at (0.5,0);
    \draw (0)--(1)--(2)--(3)--(4)--(1) (0)--(3)
    (2)--(5)--(0)--(6)--(4);
    \draw[red] (1) .. controls (.35,0) and (.35,0) .. (3) (1)--(5)
    (1)--(6);
    \draw[blue] (0) .. controls (0,.25) and (-1,.25) .. (2) (2)
    .. controls (-1,-.25) and (1,-.25) .. (4);
    \filldraw (0) circle (1pt) (1) circle (1pt) (2) circle (1pt) (3)
    circle (1pt) (4) circle (1pt) (5) circle (1pt) (6) circle (1pt);
  \end{tikzpicture}
  \caption{}\label{sf:one}
    \end{subfigure}
     \begin{subfigure}{.24\textwidth}    \centering
      \begin{tikzpicture}
    \tiny
    \coordinate[label={[label distance=-1pt] 45:$$}] (0) at (0,0);
    \coordinate[label={[label distance=-1pt] 90:$$}] (1) at (0,1);
    \coordinate[label={[label distance=-1pt] 180:$$}] (2) at (-1,0);
    \coordinate[label={[label distance=-1pt] 270:$$}] (3) at (0,-1);
    \coordinate[label={[label distance=-1pt] 0:$$}] (4) at (1,0);
    \coordinate[label={[label distance=-1pt] 90:$$}] (5) at (-.5,0);
    \coordinate[label={[label distance=-1pt] 90:$$}] (6) at (0.5,0);
    \draw (0)--(1)--(2)--(3)--(4)--(1) (0)--(3)
    (2)--(5)--(0)--(6)--(4);
        \draw[red] (1) .. controls (.35,0) and (.35,0) .. (3) (1)--(5)
        (3)--(6);
         \draw[blue] (0) .. controls (0,.25) and (-1,.25) .. (2) (2)
    .. controls (-1,-.25) and (1,-.25) .. (4);
    \filldraw (0) circle (1pt) (1) circle (1pt) (2) circle (1pt) (3)
    circle (1pt) (4) circle (1pt) (5) circle (1pt) (6) circle (1pt);
  \end{tikzpicture}
    \caption{}\label{sf:two}
     \end{subfigure}
      \begin{subfigure}{.24\textwidth}    \centering
      \begin{tikzpicture}
    \tiny
    \coordinate[label={[label distance=-1pt] 45:$$}] (0) at (0,0);
    \coordinate[label={[label distance=-1pt] 90:$$}] (1) at (0,1);
    \coordinate[label={[label distance=-1pt] 180:$$}] (2) at (-1,0);
    \coordinate[label={[label distance=-1pt] 270:$$}] (3) at (0,-1);
    \coordinate[label={[label distance=-1pt] 0:$$}] (4) at (1,0);
    \coordinate[label={[label distance=-1pt] 90:$$}] (5) at (-.5,0);
    \coordinate[label={[label distance=-1pt] 90:$$}] (6) at (0.5,0);
    \draw (0)--(1)--(2)--(3)--(4)--(1) (0)--(3)
    (2)--(5)--(0)--(6)--(4);
    \draw[red] (1) .. controls (.35,0) and (.35,0) .. (3) (1)--(5)
    (1)--(6);
    \draw[blue] (0) .. controls (0,-.25) and (-1,-.25) .. (2) (0)
    .. controls (0,.25) and (1,.25) .. (4);
    \filldraw (0) circle (1pt) (1) circle (1pt) (2) circle (1pt) (3)
    circle (1pt) (4) circle (1pt) (5) circle (1pt) (6) circle (1pt);
  \end{tikzpicture}
      \caption{}\label{sf:three}
      \end{subfigure}
       \begin{subfigure}{.24\textwidth}    \centering
      \begin{tikzpicture}
    \tiny
    \coordinate[label={[label distance=-1pt] 45:$$}] (0) at (0,0);
    \coordinate[label={[label distance=-1pt] 90:$$}] (1) at (0,1);
    \coordinate[label={[label distance=-1pt] 180:$$}] (2) at (-1,0);
    \coordinate[label={[label distance=-1pt] 270:$$}] (3) at (0,-1);
    \coordinate[label={[label distance=-1pt] 0:$$}] (4) at (1,0);
    \coordinate[label={[label distance=-1pt] 90:$$}] (5) at (-.5,0);
    \coordinate[label={[label distance=-1pt] 90:$$}] (6) at (0.5,0);
    \draw (0)--(1)--(2)--(3)--(4)--(1) (0)--(3)
    (2)--(5)--(0)--(6)--(4);
        \draw[red] (1) .. controls (.35,0) and (.35,0) .. (3) (1)--(5)
        (3)--(6);
         \draw[blue] (0) .. controls (0,-.25) and (-1,-.25) .. (2) (0)
    .. controls (0,.25) and (1,.25) .. (4);
    \filldraw (0) circle (1pt) (1) circle (1pt) (2) circle (1pt) (3)
    circle (1pt) (4) circle (1pt) (5) circle (1pt) (6) circle (1pt);
  \end{tikzpicture}
      \caption{}\label{sf:four}
    \end{subfigure}
    \caption{Some potential $\Lambda$ for $\Gamma$ from reverse coning
      of a satellite-dismantling.}
    \label{fig:potentialLambda}
  \end{figure}
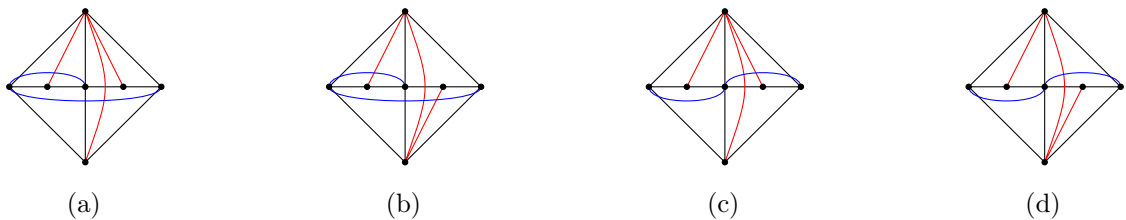

\end{exam}

\bibliographystyle{hypersshort}
\bibliography{main}

\end{document}